\documentclass[11pt]{amsart}
\usepackage{mathrsfs}
\usepackage{amssymb}

\usepackage{color}
\usepackage{amscd}
\usepackage{amsmath, amssymb}
\usepackage{amsfonts}
\usepackage[colorlinks,linkcolor=green,citecolor=red, pdfstartview=FitH]{hyperref}

\numberwithin{equation}{section}

\theoremstyle{plain}
\newtheorem{thm}{Theorem}[section]
\newtheorem{theorem}[thm]{Theorem}
\newtheorem{lemma}[thm]{Lemma}
\newtheorem{corollary}[thm]{Corollary}
\newtheorem{proposition}[thm]{Proposition}

\theoremstyle{definition}

\newtheorem{remark}[thm]{Remark}

\newtheorem{definition}[thm]{Definition}

\newtheorem{example}[thm]{Example}

\newtheorem{defn-thm}[thm]{Definition-Theorem}

\newtheorem*{ack}{Acknowledgments}

\newcommand{\G}{{\mathbb G}}
\renewcommand{\H}{{\mathbb H}}

\newcommand{\Ker}{{ Ker}}

\newcommand{\bp}{\bar{\partial}}

\newcommand{\btheorem}{\begin{theorem}}
\newcommand{\etheorem}{\end{theorem}}
\newcommand{\bproposition}{\begin{proposition}}
\newcommand{\eproposition}{\end{proposition}}
\newcommand{\bdefinition}{\begin{definition}}
\newcommand{\edefinition}{\end{definition}}
\newcommand{\bcorollary}{\begin{corollary}}
\newcommand{\ecorollary}{\end{corollary}}
\newcommand{\bproof}{\begin{proof}}
\newcommand{\eproof}{\end{proof}}
\newcommand{\bremark}{\begin{remark}}
\newcommand{\eremark}{\end{remark}}
\newcommand{\eexample}{\end{example}}
\newcommand{\bexample}{\begin{example}}
\newcommand{\la}{\langle}
\newcommand{\elemma}{\end{lemma}}
\newcommand{\blemma}{\begin{lemma}}
\newcommand{\ra}{\rangle}

\newcommand{\p}{\partial}

\renewcommand{\bar}{\overline}

\renewcommand{\phi}{\varphi}

\newcommand{\ka}{K\"ahler }

\newcommand{\ee}{\end{eqnarray*}}
\newcommand{\be}{\begin{eqnarray*}}

\newcommand{\beq}{\begin{equation}}
\newcommand{\eeq}{\end{equation}}
\newcommand{\bd}{\begin{enumerate}}
\newcommand{\ed}{\end{enumerate}}

\begin{document}
\title{Extension of holomorphic canonical forms 
on complete $d$-bounded\\ 
\ka manifolds}
\makeatletter
\let\uppercasenonmath\@gobble
\let\MakeUppercase\relax
\let\scshape\relax
\makeatother

\author{Chunle Huang}
\address{Chunle Huang, Institute of Mathematics, Hunan University, Changsha, 410082, China.}
\email{chunle@zju.edu.cn; 402961544@qq.com}
\maketitle

\begin{abstract}
In this paper we study the extension of holomorphic canonical forms on complete $d$-bounded \ka manifolds by using $L^2$ analytic methods  and $L^2$ Hogde theory, which generalizes some classical results to noncompact cases.
\end{abstract}
\maketitle
\tableofcontents

\section{Introduction}
Following the recent papers \cite{LRY, LZ} we will present some results 
in this paper about the extension of holomorphic canonical forms on complete $d$-bounded \ka manifolds,  by using $L^2$ analytic methods \cite{Dem,Demailly02,Gromov,Hormander09} and $L^2$ Hogde theory \cite{GH,KK,LZ,Zucker22}, which is related to some extent to Siu’s conjecture of the invariance of plurigenera for any compact \ka manifolds \cite{Siu98, Siu}. Part of our results generalizes some important results obtained in \cite{LRY, LZ} from the compact to noncompact cases.
Recall that a differential form $\alpha$ on a Riemannian manifold $(X, g)$ is said to be bounded with respect to the Riemannian
metric $g$ if the $L^{\infty}$-norm of $\alpha$ is finite, that is,
\begin{equation}\nonumber
\|\alpha\|_{L^{\infty}(X)}:=\sup_{x\in X}\|\alpha\|_g<\infty.
\end{equation}
Following Gromov\cite{Gromov} we say that $\alpha$ is $d$-bounded if $\alpha$ is the exterior differential of a bounded form $\beta$, that is, 
$$\alpha= d\beta\,\,\text{with}\,\,\|\beta\|_{L^{\infty}(X)}<\infty.$$
In particular, we need to introduce the following definition.
\begin{definition}[\cite{Gromov}]
A \ka manifold $(X,\omega)$ of dimension $n$ is called a $d$-bounded \ka manifold, 
if the \ka form $\omega$ is $d$-bounded, that is, there exists a bounded $1$-form $\theta$ on $X$ with respect to $g$ such that 
$\omega=d\theta$, where $g$ is the Riemannian metric indued by $\omega$. \end{definition}
For any \ka manifold $(X,\omega)$ of dimension $n$, we denote by $L_{(2)}^{p,q}(X,\omega)$ the space of $(p,q)$ forms
$\psi$ on $X$ with $0\leq p,q\leq n$ such that $\psi$ is $L^2$ global integrable with respect to $\omega$ in this paper. For the details of $L^2$ space and $L^2$ method we recommend the reader to see \cite{Dem,Demailly02,Fujino,Hormander09,Matsumura S.}.
By 1.4.A. Theorem in \cite{Gromov} we have 
\begin{theorem}\label{Key.}
If $(X,\omega)$ be a complete $d$-bounded \ka manifold, then
\begin{equation}\label{KeyPoint}\nonumber
\bar{Im \overline{\partial}}=Im \overline{\partial},\quad
\bar{Im \overline{\partial}^{*}}=Im \overline{\partial}^{*}.
\end{equation}
Moreover, $(X,\omega)$ admits the following fundamental $L^2$ Hodge theory:
\begin{enumerate}
  \item $L^2$ Hodge orthogonal decomposition :
$L^{p,q}_{(2)}(X,\omega)=\mathcal{H}_{(2)}^{p,q}(X,\omega)\oplus Im \overline{\partial}\oplus Im \overline{\partial}^{*}$
  \item $L^2$ Hodge isomorphism :
  $H_{(2)}^{p,q}(X,\omega)\cong\mathcal{H}_{(2)}^{p,q}(X,\omega)$
  \item denoting $\mathbb{H}$ to be the projection operator
  $\mathbb{H}: L^{p,q}_{(2)}(X,\omega)\rightarrow \mathcal{H}_{(2)}^{p,q}(X,\omega)$
then the Green operator $\mathbb{G}=(\Delta_{\overline{\partial}}|_{\mathcal{H}_{(2)}^{p,q}(X,\omega)^{\bot}})^{-1}(\mathbb{I}-\mathbb{H})$
is well-defined and bounded. And, we have
the following identity $$\Delta_{\overline{\partial}}\mathbb{G}=\mathbb{G}\Delta_{\overline{\partial}}=\mathbb{I}-\mathbb{H}, \mathbb{HG}=\mathbb{GH}=0$$
\end{enumerate}
\end{theorem}
From Theorem \ref{Key.} we know that the $L^2$ Hodge theory holds for any complete $d$-bounded \ka manifold $(X,\omega)$. This is very important for us since it provides us a potential Hodge method as used in \cite{LRY,LZ} to solve the extension $\bp$-equation (\ref{holo.}) from the deformation theory of Kodaira-Spencer-Kuranishi with suitable $L^2$ estimates on $(X,\omega)$.  A complete $d$-bounded \ka manifold $(X,\omega)$ is not necessarily compact, which constitutes the main difficulty in our research and can makes almost everything different from the classical compact cases. In fact, to study the extension $\bp$-equation (\ref{holo.}) on the complete $d$-bounded \ka manifold $(X,\omega)$, we need to prove some basic $L^2$ estimates on $(X,\omega)$ by applying the Hopf-Rinow lemma in terms of the operators as presented in the $L^2$ Hodge theory. In particular, we obtain a quasi-isometry formula in $L^{2}$-norm with respect to the operator $\bp^*\G\p$.

\begin{theorem}\label{po.}
Let $(X,\omega)$ be a complete $d$-bounded \ka manifold of dimension $n$. Then for any $g\in L_{(2)}^{p-1,q}(X,\omega)$ with $\p g\in L^{p,q}_{(2)}(X,\omega)$, we have $$\|\bp^*\G\p
g\|^2\leq \|g\|^2.$$
\end{theorem}
We recall some basic concepts on complex structures
and Beltrami differentials. Let M be a complex manifold of complex dimension $\dim M=n $. A Beltrami differential $\varphi$ is by definition a tangent bundle valued (0, 1)-form in $A^{0,1}(M, T^{1,0}M)$. If the Beltrami differential $\varphi$ is integrable in the sense that $$\bp\varphi=\frac{1}{2}[\varphi,\varphi]$$ then $\varphi$ determines a new complex structure on $M$, which is denoted by $M_{\varphi}$ in this paper.
By applying Theorem \ref{po.} we will study in detail the $L^2$ extension equation of holomorphic canonical forms on complete $d$-bounded \ka manifolds in Section \ref{Proof}. 
Let $(X,\omega)$ be a complete $d$-bounded \ka manifold of dimension $n$, $A_{(2)}^{p,q}(X,\omega)$ be the space of smooth $(p,q)$ forms $g$ on $X$
such that $g$ is $L^2$ global integrable with respect to $\omega$ and 
$\varphi\in A^{0,1}(X,T^{1,0}_X)$ be an integral Beltrami differential on $X$
such that its $L_\infty$-norm $\|\varphi\|_{\omega,\infty}$ with respect to $\omega$ is less than 1,
that is, $\|\varphi\|_{\omega,\infty}<1.$
Then for any $g\in A_{(2)}^{p,q}(X,\omega)$ we have
\begin{equation*}
  \|\varphi\lrcorner g\|_\omega\leq\|\varphi\|_{\omega,\infty}\|g\|_\omega\leq\|g\|_\omega<\infty.
\end{equation*}
Following the paper \cite{LZ} we consider the operator 
\begin{equation}\nonumber
T: L^{p,q}_{(2)}(X,\omega)\rightarrow L^{p+1,q-1}_{(2)}(X,\omega)
\end{equation}
defined by
\begin{equation}\nonumber
  T=\bp ^*\mathbb{G}\p.
\end{equation}
By Theorem \ref{po.} we will show that the operator $I+T\varphi$ is injective. It follows that the deformation operator 
$$\rho_{\omega, \varphi}: \mathcal{A}^{n,0}_{(2)}(X,\omega,\varphi)\rightarrow A^{n,0}(X_{\varphi})$$defined by 
\begin{equation}
\rho_{\omega, \varphi}(\Omega)=e^{\varphi}\lrcorner(I+T\varphi)^{-1}\Omega,\,\,\Omega\in\mathcal{A}^{n,0}_{(2)}(X,\omega,\varphi)\nonumber
\end{equation}
is well-defined, where
$$\mathcal{A}^{n,0}_{(2)}(X,\omega,\varphi):=
Im(I+T\varphi)\subset A^{n,0}_{(2)}(X,\omega).
$$
Finally, by using our $L^2$ estimates on the complete $d$-bounded \ka manifold $(X,\omega)$ and some analysis of the $L^2$ extension equation (\ref{holo..}), see the following section, we have the following main result about extensions of holomorphic canonical forms from the complex manifold $X$ to complex manifold $X_\varphi$.
\btheorem\label{.main}
Let $(X,\omega)$ be a complete $d$-bounded \ka manifold of dimension $n$ and $\varphi\in A^{0,1}(X,T^{1,0}_X)$
be an integral Beltrami differential on $X$ such that 
$\|\varphi\|_{\omega,\infty}<1$.
Then for any holomorphic $(n,0)$-form $\Omega$ in $\mathcal{A}^{n,0}_{(2)}(X,\omega,\varphi)$,
the expression $\rho_{\omega,\varphi}(\Omega)$
defines a holomorphic $(n,0)$-form on $X_\varphi$ with $\rho_{\omega,0}(\Omega)=\Omega$.\etheorem
We say that the canonical  holomorphic $(n,0)$-form $\rho_{\omega,\varphi}(\Omega)$ on $X_\varphi$ is a holomorphic extension of the canonical  holomorphic $(n,0)$-form $\Omega$ on $X$ in this paper. Theorem \ref{.main} tell us that, for any complete $d$-bounded \ka manifold $(X,\omega)$ of dimension $n$, there always exists a subspace $\mathcal{A}^{n,0}_{(2)}(X,\omega,\varphi)$ of  $A^{n,0}_{(2)}(X,\omega)$ such that every holomorphic $(n,0)$-form $\Omega$ in $\mathcal{A}^{n,0}_{(2)}(X,\omega,\varphi)$ admits a holomorphic extension $\rho_{\omega,\varphi}(\Omega)$ from $X$ to $X_{\varphi}$ with $\rho_{\omega,0}(\Omega)=\Omega$. This result generalizes Theorem 1.1 in \cite{LZ} from the compact to noncompact cases, which is closely related to a famous conjecture due to Siu \cite{Siu98, Siu}, about the invariance of plurigenera for compact \ka manifolds.

 \begin{ack}
The author would like to express his sincerely gratitude to Prof.
Kefeng Liu for continued support and interest on this work.
\end{ack}

\section{Proof of Theorem \ref{.main}}\label{Proof}
In this section we give a detailed proof of Theorem \ref{.main}.  
First we need the follow estimate from \cite{Gromov} which plays a key role in building the $L^2$-Hodge theory on complete $d$-bounded \ka manifolds.
\begin{lemma}[1.4.A. Theorem\cite{Gromov}]\label{KeyLemma}
Let $(X,\omega)$ be a complete $d$-bounded \ka manifold of dimension $n=2m$ and $\omega=d\eta$ where $\eta$ is a bounded $1$-form on $X$. Then every $L^2$-form $\psi$ on $X$ of degree $p\neq m$ satisfies the inequality 
\begin{equation}\label{Gro.}
\la\la\psi,\Delta\psi\ra\ra\geq\lambda^2\la\la\psi,\psi\ra\ra
\end{equation}
where $\lambda$ is strictly positive constant which depends only on $n=\dim X$ and the bound on $\eta$. Furthermore, inequality in the $($\ref{Gro.}$)$ is satisfied by the $L^2$-forms of degree $m$ which are orthogonal to the harmonic $m$-forms. Here every term in $($\ref{Gro.}$)$ is allowed to be infinity.
\end{lemma}
For any \ka manifold $(X,\omega)$ we denote by $L_{(2)}^{p,q}(X,\omega)$ the space of $(p,q)$ forms $\psi$ on $X$ such that $\psi$ is $L^2$ global integrable with respect to the metric $\omega$. As an important consequence of Lemma \ref{KeyLemma} we obtain 
\begin{theorem}[=Theorem \ref{Key.}]\label{Key}
If $(X,\omega)$ be a complete $d$-bounded \ka manifold, then we have
\begin{equation}\label{KeyPoint}\nonumber
\bar{Im \overline{\partial}}=Im \overline{\partial},\quad
\bar{Im \overline{\partial}^{*}}=Im \overline{\partial}^{*}.
\end{equation}
Moreover, $(X,\omega)$ admits the following fundamental $L^2$ Hodge theory:
\begin{enumerate}
  \item $L^2$ Hodge orthogonal decomposition :
$L^{p,q}_{(2)}(X,\omega)=\mathcal{H}_{(2)}^{p,q}(X,\omega)\oplus Im \overline{\partial}\oplus Im \overline{\partial}^{*},$
  \item $L^2$ Hodge isomorphism :
  $H_{(2)}^{p,q}(X,\omega)\cong\mathcal{H}_{(2)}^{p,q}(X,\omega).$
  \item denoting $\mathbb{H}$ to be the projection operator
  $\mathbb{H}: L^{p,q}_{(2)}(X,\omega)\rightarrow \mathcal{H}_{(2)}^{p,q}(X,\omega)$
then the Green operator $\mathbb{G}=(\Delta_{\overline{\partial}}|_{\mathcal{H}_{(2)}^{p,q}(X,\omega)^{\bot}})^{-1}(\mathbb{I}-\mathbb{H})$
is well-defined and bounded. And, we have
the following identity $$\Delta_{\overline{\partial}}\mathbb{G}=\mathbb{G}\Delta_{\overline{\partial}}=\mathbb{I}-\mathbb{H}, \mathbb{HG}=\mathbb{GH}=0.$$
\end{enumerate}
\end{theorem}
\begin{proof}
From the arguments in the proof of l.l.B.Lemma \cite{Gromov} we have 
\begin{equation}\nonumber
\la\la\psi,\Delta\psi\ra\ra=\la\la\bar\partial\psi,\bar\partial\psi\ra\ra+\la\la\bar\partial^*\psi,\bar\partial^*\psi\ra\ra
=\|\bar\partial\psi\|^2+\|\bar\partial^*\psi\|^2
\end{equation}
for any $\psi\in \text{Dom}\overline{\partial}\cap \text{Dom}\bar\partial^*\cap(\Ker\bp\cap\Ker\bp^*)^{\perp}$.
It follows from Lemma \ref{KeyLemma} that 
\begin{equation}\nonumber
\begin{split}
\lambda^2\|\psi\|^2&=\lambda^2\la\la\psi,\psi\ra\ra\\
&\leq\la\la\psi,\Delta\psi\ra\ra\\
&=\|\bar\partial\psi\|^2+\|\bar\partial^*\psi\|^2\\
&\leq(\|\bar\partial\psi\|+\|\bar\partial^*\psi\|)^2
\end{split}
\end{equation}
which yields that
\begin{equation}\nonumber
\|\psi\|\leq\lambda^{-1}(\|\bar\partial\psi\|+\|\bar\partial^*\psi\|)
\end{equation}
for any $\psi\in \text{Dom}\overline{\partial}\cap \text{Dom}\bar\partial^*\cap(\Ker\bp\cap\Ker\bp^*)^{\perp}$. 
Then Theorem \ref{Key} follows from
Theorem 2.1 and Theorem 2.2 in Chapter 2.1.2 of \cite{Ohsawa},
pages 42-43.
\end{proof}
Let $(X,\omega)$ be a complete $d$-bounded \ka manifold of dimension $n$. We will prove some basic $L^2$ estimates on $(X,\omega)$ in terms of the operators as in Theorem \ref{Key}, in particular a quasi-isometry formula, that is, Theorem \ref{po} as below,
in $L^{2}$-norm with respect to the operator $\bp^*\G\p$.
As one application, we will see that these estimates give a rather
simple and explicit way to solve some $\bp$-equations with suitable $L^2$-estimates on $(X,\omega)$, see Theorem \ref{eq}.
As another application of our estimates, we will apply these estimates to studying the $L^2$ extension equation
of holomorphic canonical forms on the complete $d$-bounded \ka manifold $(X,\omega)$ as follows. To begin, we need the following Hopf-Rinow lemma, 
which is very effective when we deal with $L^2$ estimates on complete manifolds.

\blemma [cf. page 366 in \cite{Demailly02}]\label{qsef.}
The following properties are equivalent:
\begin{itemize}
  \item [$(a)$] $(M, g)$ is complete;
  \item [$(b)$] there exists an exhaustive sequence $\{K_\nu\}_{\nu\in \mathbb{N}}$ of compact subsets of $M$ and functions
 $\psi_\nu\in C^\infty(M,\mathbb{R})$ such that
  $\psi_\nu=1$ in a neighborhood of $K_\nu$, $\text{Supp}\psi_\nu\subset  K^\circ_{\nu+1}$,
$ 0\leq \psi_\nu\leq 1$ and $|d\psi_\nu|_g \leq 2^{-\nu}$.
\end{itemize}
\elemma

By applying Lemma \ref{qsef.} we obtain
\blemma \label{fdsa}
Let $(X,\omega)$ be a complete $d$-bounded \ka manifold of dimension $n$. Then for any $g\in L_{(2)}^{p,q}(X,\omega)$ with $q\geq1$ 
$$\|\bp ^*\mathbb{G}g\|^2\leq \la\la  g, \mathbb{G} g\ra\ra\leq \|\mathbb{G}\|\cdot\|g\|^2.$$
\elemma
\bproof
Since $\|\bp ^*\mathbb{G}g\|^2=\lim_{\nu\rightarrow\infty}\|\psi_\nu\bp ^*\mathbb{G}g\|^2$ and
\begin{equation}
\begin{split}
 \|\psi_\nu\bp ^*\mathbb{G}g\|^2&=\lim_{\nu\rightarrow\infty}\la\la \psi_\nu^2 \bp ^*\mathbb{G}g, \bp ^*\mathbb{G}g\ra\ra\\
  &=\lim_{\nu\rightarrow\infty}\la\la \bp(\psi_\nu^2 \bp ^*\mathbb{G}g), \mathbb{G}g\ra\ra\\
  &=\lim_{\nu\rightarrow\infty}\la\la 2\psi_\nu\bp\psi_\nu\wedge \bp ^*\mathbb{G}g, \mathbb{G}g\ra\ra+\lim_{\nu\rightarrow\infty}\la\la \psi_\nu^2 \bp\bp ^*\mathbb{G}g, \mathbb{G}g\ra\ra\nonumber
\end{split}
\end{equation}
it follows that
\begin{equation}\label{qw}
  \lim_{\nu\rightarrow\infty}\|\psi_\nu\bp ^*\mathbb{G}g\|^2=\lim_{\nu\rightarrow\infty} \la\la \psi_\nu^2 \bp\bp ^*\mathbb{G}g, \mathbb{G}g\ra\ra.
\end{equation}
Here we used the fact that
\begin{equation}
\begin{split}
 \lim_{\nu\rightarrow\infty}|\la\la 2\psi_\nu\bp\psi_\nu\wedge \bp ^*\mathbb{G}g, \mathbb{G}g\ra\ra|&\leq
 \lim_{\nu\rightarrow\infty}\| \bp\psi_\nu\|\| 2\psi_\nu\bp ^*\mathbb{G}g\|\| \mathbb{G}g\|\\
 &\leq\lim_{\nu\rightarrow\infty}\| \bp\psi_\nu\|\| 2\bp ^*\mathbb{G}g\|\| \mathbb{G}g\|=0.\nonumber
\end{split}
\end{equation}
We compute
\begin{equation}\label{aaty}
\begin{split}
\lim_{\nu\rightarrow\infty}\la\la \psi_\nu^2 \bp\bp ^*\mathbb{G} g, \mathbb{G} g\ra\ra)
&=\lim_{\nu\rightarrow\infty}\la\la \psi_\nu^2 (\Delta_{\overline{\partial}} \mathbb G-\bp ^*\bp\mathbb{G}) g, \mathbb{G} g\ra\ra\\
&=\lim_{\nu\rightarrow\infty}\la\la \psi_\nu^2 (\mathbb{I}-\mathbb H-\bp ^*\bp\mathbb{G}) g, \mathbb{G} g\ra\ra\\
&=\lim_{\nu\rightarrow\infty}\la\la \psi_\nu^2  g, \mathbb{G} g\ra\ra-\lim_{\nu\rightarrow\infty}\la\la \psi_\nu^2 \mathbb H g, \mathbb{G} g\ra\ra
-\lim_{\nu\rightarrow\infty}\la\la \psi_\nu^2 \bp ^*\bp\mathbb G g, \mathbb{G} g\ra\ra\\
&=\la\la   g, \mathbb{G} g\ra\ra-\la\la  \mathbb H g, \mathbb{G} g\ra\ra
-\lim_{\nu\rightarrow\infty}\la\la \psi_\nu^2 \bp ^*\bp\mathbb G g, \mathbb{G} g\ra\ra
\end{split}
\end{equation}
Note that $\la\la  \mathbb H g, \mathbb{G} g\ra\ra=\la\la  \mathbb{G}\mathbb H g,  g\ra\ra=0$
since the Green operator is self-adjoint and zero on the kernel of
Laplacian by definition. Moreover,
\begin{equation}
  \begin{split}
   \la\la \psi_\nu^2 \bp ^*\bp\mathbb G g, \mathbb{G} g\ra\ra&=\la\la  \bp ^*\bp\mathbb G g, \psi_\nu^2\mathbb{G}g\ra\ra\\
   &=\la\la  \bp\mathbb G g, \bp (\psi_\nu^2\mathbb{G} g)\ra\ra\\
   &=\la\la  \bp\mathbb G g, 2\psi_\nu\bp\psi_\nu\wedge\mathbb{G} g\ra\ra+
   \la\la  \bp\mathbb G g, \psi_\nu^2\bp\mathbb{G} g\ra\ra.\nonumber
  \end{split}
\end{equation}
But
\begin{equation}
\begin{split}
  |\la\la  \bp\mathbb G g, 2\psi_\nu\bp\psi_\nu\wedge\mathbb{G} g\ra\ra|&\leq\|\bp\psi_\nu\| \|2\psi_\nu\bp\mathbb G g\| \|\mathbb{G} g\|\\
 &\leq\|\bp\psi_\nu\| \|2\bp\mathbb G g\| \|\mathbb{G} g\|\\
 &\leq2^{-\nu} \|2\bp\mathbb G g\| \|\mathbb{G} g\|\rightarrow0.\nonumber
\end{split}
\end{equation}
Therefore
\begin{equation}\label{aaio}
 \lim_{\nu\rightarrow\infty}\la\la \psi_\nu^2 \bp ^*\bp\mathbb G g, \mathbb{G} g\ra\ra=\|\bp\mathbb{G} g\|^2. \nonumber
\end{equation}
It follows from (\ref{qw}) and (\ref{aaty}) that
\begin{equation}
  \begin{split}
    \|\bp ^*\mathbb{G}g\|^2&=\lim_{\nu\rightarrow\infty}\|\psi_\nu\bp ^*\mathbb{G}g\|^2
    =\lim_{\nu\rightarrow\infty} \la\la \psi_\nu^2 \bp\bp ^*\mathbb{G}g, \mathbb{G}g\ra\ra\\
    &=\la\la   g, \mathbb{G} g\ra\ra-\|\bp\mathbb{G} g\|^2
    \leq \la\la   g, \mathbb{G} g\ra\ra \nonumber
    \end{split}
\end{equation}
This completes the proof.
\eproof

As an easy application of Theorem \ref{Key} and Lemma \ref{fdsa}, we have
\btheorem \label{eq}
Let $(X,\omega)$ be a complete $d$-bounded \ka manifold of dimension $n$. Then for any $g\in L_{(2)}^{p-1,q}(X,\omega)$ with $\p g\in L^{p,q}_{(2)}(X,\omega)$
and $q\geq1$, the differential form
$$s = {\bp}^* \mathbb{G} \p g\in L^{p,q-1}_{(2)}(X,\omega)$$ gives a solution to the equation
\begin{equation}\label{dfmt}
 {\bp } s =\p g
\end{equation}
with $\bp  \p g =0$ and $$\|s\|^2\leq \la\la
\p g, \G  \p g\ra\ra\leq\|G\|\|\p g\|^2.$$ This solution is unique if we require
$\mathbb{H}(s)=0$ and $\overline{\partial}^{\ast}s=0$.
\etheorem

\begin{proof}
By the $L^2$ Hodge decomposition, see Theorem \ref{Key}, we have
$$\bp  s = {\bp } {\bp }^* \mathbb{G} \p g
= \p g-\mathbb{H}\p g-{\bp }^* {\bp } \mathbb{G} \p g=
\p g-\mathbb{H}\p g= \p g,$$ where we used the identity
$\mathbb{H}\p g=0$. The estimate in Theorem \ref{eq} follows from Lemma \ref{fdsa}.
The uniqueness of this solution is obvious.  In fact, if $s_1$
and $s_2$ are two solutions to $\bp s=\p  g$ with
$\H(s_1)=\H(s_2)=0$ and $\bp^*s_1=\bp^* s_2=0$, by setting
$\eta=s_1-s_2$, we see $\bp\eta =0$, $\H(\eta)=0$ and $\bp^*\eta=0$.
Therefore,
\begin{align*}
\eta=\mathbb{H}(\eta)+\Delta_{\overline{\partial}}\mathbb{G}(\eta)
=\mathbb{H}(\eta)+(\overline
{\partial}\overline{\partial}^{*}+\overline{\partial}^{*}\overline{\partial
})\mathbb{G}(\eta) =0,
\end{align*}
as desired.
\end{proof}

We remark that such kind of $\bp$-equation as in formula (\ref{dfmt}) is very important
in the study of the holomorphic deformation theory of Kodaira-Spencer-Kuranishi
since its solution can be used to construct holomorphic deformations
of complex structures (cf. \cite{GLT,MK,LSY,[T],[To89]}).
Theorem \ref{eq} gives a rather simple and explicit way to solve
this kind of $\bp$-equation with suitable $L^2$-estimates on complete \ka manifolds,
which generalizes Proposition 2.3 in \cite{LRY}.

Next by using Lemma \ref{qsef.} we have the following quasi-isometry formula
in $L^{2}$-norm with respect to the operator $\bp^*\G\p$.
We will apply it to studying the $L^2$ extension equation, from the deformation theory of Kodaira-Spencer-Kuranishi, of holomorphic canonical forms on complete $d$-bounded \ka manifolds.

\begin{theorem}[=Theorem \ref{po.}]\label{po}
Let $(X,\omega)$ be a complete $d$-bounded \ka manifold of dimension $n$. Then for any $g\in L_{(2)}^{p-1,q}(X,\omega)$ with $\p g\in L^{p,q}_{(2)}(X,\omega)$, we have $$\|\bp^*\G\p
g\|^2\leq \|g\|^2.$$
\end{theorem}
\bproof
First, for any $g\in L_{(2)}^{p-1,q}(X,\omega)$ with $\p g\in L^{p,q}_{(2)}(X,\omega)$, we find
$\Delta_{\overline{\partial}} \mathbb G \p g=\mathbb I \p g-\mathbb{H}\p g\,\,\in L^{p,q}_{(2)}(X,\omega),$
in particular,
\begin{equation}\label{ssjk}
  \mathbb G \p g\,\,\in \text{Dom}\Delta_{\overline{\partial}}\subset\text{Dom}\bp\cap \text{Dom}\bp^*
\end{equation}
and
\begin{equation}\label{nnas}
  \mathbb G \p g\,\,\in \text{Dom}\Delta_{\overline{\partial}}
=\text{Dom}\Delta_{\partial}\subset\text{Dom}\p\cap \text{Dom}\p^*.
\end{equation}
Next, we notice that \begin{equation}
\begin{split}
 \|\psi_\nu\bp ^*\mathbb{G}\p g\|^2&=\la\la \psi_\nu^2 \bp ^*\mathbb{G}\p g, \bp ^*\mathbb{G}\p g\ra\ra\\
  &=\la\la \bp(\psi_\nu^2 \bp ^*\mathbb{G}\p g), \mathbb{G}\p g\ra\ra\\
  &=\la\la 2\psi_\nu\bp\psi_\nu\wedge \bp ^*\mathbb{G}\p g, \mathbb{G}\p g\ra\ra+\la\la \psi_\nu^2 \bp\bp ^*\mathbb{G}\p g, \mathbb{G}\p g\ra\ra\\
  &\leq 2^{-\nu} \| 2\psi_\nu\bp ^*\mathbb{G}\p g\|\| \mathbb{G}\p g\|+\la\la \psi_\nu^2 \bp\bp ^*\mathbb{G}\p g, \mathbb{G}\p g\ra\ra\\
  &\leq 2^{-\nu} \big(\| \psi_\nu\bp ^*\mathbb{G}\p g\|^2+\| \mathbb{G}\p g\|^2\big)+\la\la \psi_\nu^2 \bp\bp ^*\mathbb{G}\p g, \mathbb{G}\p g\ra\ra\\
  &\leq 2^{-\nu} \| \psi_\nu\bp ^*\mathbb{G}\p g\|^2+2^{-\nu}\| \mathbb{G}\p g\|^2+\la\la \psi_\nu^2 \bp\bp ^*\mathbb{G}\p g, \mathbb{G}\p g\ra\ra.\nonumber
  \end{split}
\end{equation}
Thus we obtain
\begin{equation}\label{er}
  \|\psi_\nu\bp ^*\mathbb{G}\p g\|^2\leq\frac{1}{1-2^{-\nu}}\Big(2^{-\nu}\| \mathbb{G}\p g\|^2+
  \la\la \psi_\nu^2 \bp\bp ^*\mathbb{G}\p g, \mathbb{G}\p g\ra\ra\Big).
\end{equation}
In the following we give a detailed estimate of the term
$\la\la \psi_\nu^2 \bp\bp ^*\mathbb{G}\p g, \mathbb{G}\p g\ra\ra$ in the above inequality
by using the $L^2$ Hodge decomposition. First we compute
\begin{equation}\label{ty}
\begin{split}
\la\la \psi_\nu^2 \bp\bp ^*\mathbb{G}\p g, \mathbb{G}\p g\ra\ra)
&=\la\la \psi_\nu^2 (\Delta_{\overline{\partial}} \mathbb G-\bp ^*\bp\mathbb{G})\p g, \mathbb{G}\p g\ra\ra\\
&=\la\la \psi_\nu^2 (\mathbb{I}-\mathbb H-\bp ^*\bp\mathbb{G})\p g, \mathbb{G}\p g\ra\ra\\
&=\la\la \psi_\nu^2 (\p g-\bp ^*\bp\mathbb{G}\p g), \mathbb{G}\p g\ra\ra\\
&=\la\la \psi_\nu^2 \p g, \mathbb{G}\p g\ra\ra-\la\la \psi_\nu^2 \bp ^*\bp\mathbb G\p g, \mathbb{G}\p g\ra\ra.
\end{split}
\end{equation}
On one hand,
\begin{equation}
  \begin{split}
  \la\la \psi_\nu^2 \p g, \mathbb{G}\p g\ra\ra&=\la\la \p(\psi_\nu^2 g)-2\psi_\nu \p \psi_\nu\wedge g, \mathbb{G}\p g\ra\ra\\
  &=\la\la \psi_\nu^2 g, \p^*\mathbb{G}\p g\ra\ra-\la\la 2\psi_\nu \p \psi_\nu\wedge g, \mathbb{G}\p g\ra\ra\\
  &=\la\la \psi_\nu^2 g, g-\mathbb{H}g-\p\p^*\mathbb{G} g\ra\ra-\la\la 2\psi_\nu \p \psi_\nu\wedge g, \mathbb{G}\p g\ra\ra\\
  &=\la\la \psi_\nu^2 g, g\ra\ra-\la\la \psi_\nu^2 g, \mathbb{H}g\ra\ra-\la\la \psi_\nu^2 g, \p\p^*\mathbb{G} g\ra\ra
  -\la\la 2\psi_\nu \p \psi_\nu\wedge g, \mathbb{G}\p g\ra\ra\nonumber
\end{split}
\end{equation}
from which we see that
\begin{equation}
  \begin{split}
  &\lim_{\nu\rightarrow\infty}\la\la \psi_\nu^2 \p g, \mathbb{G}\p g\ra\ra\\
  &=\la\la g, g\ra\ra-\la\la g, \mathbb{H}g\ra\ra-\la\la g, \p\p^*\mathbb{G} g\ra\ra
  -\lim_{\nu\rightarrow\infty}\la\la 2\psi_\nu \p \psi_\nu\wedge g, \mathbb{G}\p g\ra\ra.\nonumber
\end{split}
\end{equation}
We notice that
\begin{equation}
  \begin{split}
 \la\la g, \p\p^*\mathbb{G} g\ra\ra
 &= \lim_{k\rightarrow\infty}\la\la g_k, \p\p^*\mathbb{G} g\ra\ra\\
  &=\lim_{k\rightarrow\infty} \la\la \p^*g_k, \p^*\mathbb{G} g\ra\ra\\
  &= \la\la \p^*g, \p^*\mathbb{G} g\ra\ra\geq0\nonumber
\end{split}
\end{equation}
for some sequence $\{g_k\}_{k=1}^{\infty}$ with compact support in $X$ such that
$\p^*g_k\rightarrow \p^*g$ with respect to the weak $L^2$-topology thanks to the completeness of $X$.
We also notice that
\begin{equation}
\begin{split}
  |\la\la 2\psi_\nu \p \psi_\nu\wedge g, \mathbb{G}\p g\ra\ra|&\leq\|\p \psi_\nu\|\|g\|\|2\psi_\nu\mathbb{G}\p g\|\\
  &\leq\|\p \psi_\nu\|(\|g\|^2+\|\psi_\nu\mathbb{G}\p g\|^2)\\
  &\leq2^{-\nu}(\|g\|^2+\|\psi_\nu\mathbb{G}\p g\|^2)\\
  &\leq2^{-\nu}(\|g\|^2+\|\mathbb{G}\p g\|^2)\rightarrow0.\nonumber
  \end{split}
\end{equation}
Thus we get
\begin{equation}\label{yu}
  \begin{split}
  \lim_{\nu\rightarrow\infty}\la\la \psi_\nu^2 \p g, \mathbb{G}\p g\ra\ra
  &\leq\|g\|^2-\lim_{\nu\rightarrow\infty}\la\la 2\psi_\nu \p \psi_\nu\wedge g, \mathbb{G}\p g\ra\ra=\|g\|^2.
\end{split}
\end{equation}
On the other hand,
\begin{equation}
  \begin{split}
   \la\la \psi_\nu^2 \bp ^*\bp\mathbb G\p g, \mathbb{G}\p g\ra\ra&=\la\la  \bp ^*\bp\mathbb G\p g, \psi_\nu^2\mathbb{G}\p g\ra\ra\\
   &=\la\la  \bp\mathbb G\p g, \bp (\psi_\nu^2\mathbb{G}\p g)\ra\ra\\
   &=\la\la  \bp\mathbb G\p g, 2\psi_\nu\bp\psi_\nu\wedge\mathbb{G}\p g\ra\ra+\la\la  \bp\mathbb G\p g, \psi_\nu^2\bp\mathbb{G}\p g\ra\ra.\nonumber
  \end{split}
\end{equation}
But note that
\begin{equation}
\begin{split}
  |\la\la  \bp\mathbb G\p g, 2\psi_\nu\bp\psi_\nu\wedge\mathbb{G}\p g\ra\ra|&\leq\|\bp\psi_\nu\| \|2\psi_\nu\bp\mathbb G\p g\| \|\mathbb{G}\p g\|\\
 &\leq\|\bp\psi_\nu\| \|2\bp\mathbb G\p g\| \|\mathbb{G}\p g\|\\
 &\leq2^{-\nu} \|2\bp\mathbb G\p g\| \|\mathbb{G}\p g\|\rightarrow0.\nonumber
\end{split}
\end{equation}
Therefore
\begin{equation}\label{io}
 \lim_{\nu\rightarrow\infty}\la\la \psi_\nu^2 \bp ^*\bp\mathbb G\p g, \mathbb{G}\p g\ra\ra=\|\bp\mathbb{G}\p g\|^2.
\end{equation}
Now, from (\ref{er}), (\ref{ty}), (\ref{yu}) and (\ref{io}) we see that
\begin{equation}
  \begin{split}
  \|\bp ^*\mathbb{G}\p g\|^2 &=\lim_{\nu\rightarrow\infty}\|\psi_\nu\bp ^*\mathbb{G}\p g\|^2\\
  &\leq\lim_{\nu\rightarrow\infty}\frac{1}{1-2^{-\nu}}\Big(2^{-\nu}\| \mathbb{G}\p g\|^2+\la\la \psi_\nu^2 \bp\bp ^*\mathbb{G}\p g, \mathbb{G}\p g\ra\ra\Big)\\
  &=\lim_{\nu\rightarrow\infty}\la\la \psi_\nu^2 \bp\bp ^*\mathbb{G}\p g, \mathbb{G}\p g\ra\ra\\
  &=\lim_{\nu\rightarrow\infty}\la\la \psi_\nu^2 \p g, \mathbb{G}\p g\ra\ra-\lim_{\nu\rightarrow\infty}\la\la \psi_\nu^2 \bp ^*\bp\mathbb G\p g, \mathbb{G}\p g\ra\ra\\
  &\leq\|g\|^2- \|\bp\mathbb{G}\p g\|^2\leq\|g\|^2\nonumber
  \end{split}
\end{equation}
where we used $\lim_{\nu\rightarrow\infty}\frac{1}{1-2^{-\nu}}2^{-\nu}\| \mathbb{G}\p g\|^2=0.$
This completes the proof.
\eproof

Let $(X,\omega)$ be a complete $d$-bounded \ka manifold of dimension $n$. We want to use Theorem \ref{po} to study the problem of extension of holomorphic canonical forms on the complete $d$-bounded \ka manifold $(X,\omega)$ in the following.
For this purpose, we let $A_{(2)}^{p,q}(X,\omega)$ be the space of smooth $(p,q)$ forms $g$ on $X$
such that $g$ is $L^2$ global integrable with respect to $\omega$ and
let $\varphi\in A^{0,1}(X,T^{1,0}_X)$ be an integral Beltrami differential on $X$
such that its $L_\infty$-norm $\|\varphi\|_{\omega,\infty}$ with respect to $\omega$ is less than 1,
that is, $$\|\varphi\|_{\omega,\infty}<1.$$
Then for any $g\in A_{(2)}^{p,q}(X,\omega)$ we have
\begin{equation*}
  \|\varphi\lrcorner g\|_\omega\leq\|\varphi\|_{\omega,\infty}\|g\|_\omega\leq\|g\|_\omega<\infty.
\end{equation*}
Following the paper \cite{LZ} we consider the operator 
\begin{equation}\nonumber
T: L^{p,q}_{(2)}(X,\omega)\rightarrow L^{p+1,q-1}_{(2)}(X,\omega)
\end{equation}
defined by
\begin{equation}\nonumber
  T=\bp ^*\mathbb{G}\p.
\end{equation}
The domain $\text{Dom}T$ of $T$ by definition is
\begin{equation*}
 \text{Dom}T =\{g\in L^{p,q}_{(2)}(X,\omega)\,\,|\,\,Tg\in L^{p+1,q-1}_{(2)}(X,\omega)\},
\end{equation*}
which, obviously, coincides with the domain $\text{Dom}\p$ of $\p$, that is,
\begin{equation}\label{ggcv}
\text{Dom}T=\text{Dom}\p = \{g\in L^{p,q}_{(2)}(X,\omega)\,\,|\,\,\p g\in L^{p+1,q-1}_{(2)}(X,\omega)\}.
\end{equation}
For the details of (\ref{ggcv}) see formulas (\ref{ssjk}) and (\ref{nnas}).
By Theorem \ref{po} we have that $T$ is an operator of norm less than or equal to 1
in the Hilbert space of $L^2$ forms,
that is, for any $g\in L_{(2)}^{p-1,q}(X,\omega)$ with $\p g\in L^{p,q}_{(2)}(X,\omega)$,
$$\|Tg\|_\omega=\|\bp^*\G\p g\|_\omega\leq \|g\|_\omega.$$
In particular, we have
\blemma\label{inj}
Let $(X,\omega)$ be a complete $d$-bounded \ka manifold of dimension $n$ and $\varphi\in A^{0,1}(X,T^{1,0}_X)$
be an integral Beltrami differential on $X$ such that 
$\|\varphi\|_{\omega,\infty}<1$.
Then the operator $I+T\varphi$ is injective on the domain $\text{Dom}T\varphi\cap A^{p,q}_{(2)}(X,\omega)$.
\elemma
\bproof
For any nonzero $x\in \text{Dom}T\varphi\cap A^{p,q}_{(2)}(X,\omega)$,
we have
\begin{equation}
  \|T\varphi x\|\leq\|\varphi x\|\leq\|\varphi\|_{\omega,\infty}\|x\|<\|x\|
\end{equation}
by Theorem \ref{po}.
Let $(I+T\varphi)x=0$, where $x\in\text{Dom}T\varphi\cap A^{p,q}_{(2)}(X,\omega)$.
If $x\neq0$ then we find
\begin{equation*}
  0=\|(I+T\varphi)x\|\geq\|x\|-\|T\varphi\|>\|x\|-\|x\|=0,
\end{equation*}
a contradiction. Therefore the equation $(I+T\varphi)x=0$ has only zero solution,
which gives the injectivity of $I+T\varphi$.
\eproof

In order to study the extension of holomorphic canonical forms on the complete $d$-bounded \ka manifold $(X,\omega)$,
we need to restrict the operator $I+T\varphi$
to the domain $\text{Dom}T\varphi\cap A^{p,q}_{(2)}(X,\omega)$.
Thus, in the following, we only consider the case that the operator
$I+T\varphi$ is defined by
$$I+T\varphi:\,\,\text{Dom}T\varphi\cap A^{p,q}_{(2)}(X,\omega)\rightarrow A^{p,q}_{(2)}(X,\omega).$$
Then the operator $I+T\varphi$ is injective by Lemma \ref{inj} and
$$\text{Im}(I+T\varphi)=(I+T\varphi)(\text{Dom}T\varphi\cap A^{p,q}_{(2)}(X,\omega)).$$
Here we need some basic results from classical deformation theory.
For the details, we recommend the reader to see \cite{GLT,LR,LRY,LZ}.
\begin{lemma} [\cite{LRY,LZ}]\label{holo}
Let $M$ be a complex manifold of dimension $n$ and let
$\varphi\in A^{0,1}(M,T^{1,0}_M)$ be an integral Beltrami differential.
Then for any smooth $(n,0)$-form $\Omega\in A^{n,0}(X)$, the corresponding $(n,0)$-form
$\rho_\varphi (\Omega)=e^{\varphi}\lrcorner\Omega$ on $M_\varphi$
is holomorphic if and only if
\begin{equation}\label{holo.}
\overline{\partial}\Omega=-\partial(\varphi\lrcorner\Omega).
\end{equation}
\end{lemma}
\begin{remark}
Equation (\ref{holo.}) is called the extension equation since its solution
can be used to construct the extensions of holomorphic $(n,0)$-forms
from complex manifold $M$ to complex manifold $M_\varphi$.
\end{remark}
Let $(X,\omega)$ be a complete $d$-bounded \ka manifold of dimension $n$. Then the complete \ka manifold $(X,\omega)$ admits $L^2$ Hodge theory, see Theorem \ref{Key} as above.
Let $\varphi\in A^{0,1}(X,T^{1,0}_X)$ be an integral Beltrami differential on $X$.
We want to apply the $L^2$ Hodge theory on $X$ to solving the equation (\ref{holo.}).
Thus we consider naturally the extension equation in $L^2$ sense,
that is, we want to find $L^2$ and smooth $(n,0)$-forms $\Omega\in A_{(2)}^{n,0}(X,\omega)$ such that
\begin{equation}\label{holo..}
\overline{\partial}\Omega=-\partial(\varphi\lrcorner\Omega),
\,\,\text{with}\,\, \Omega\in \text{Dom}\bp\,\,\text{and}\,\,\varphi\lrcorner\Omega\in \text{Dom}\p.
\end{equation}
By Lemma \ref{holo} the solution of equation (\ref{holo..}) can be used
to construct extensions of holomorphic $(n,0)$-forms from $X$ to $X_\varphi$.
Here we should point out that since the complex manifold $X$
 we considered in this paper may not be compact
we can not directly use Hodge theory as in \cite{LRY,LZ} to solve the equation (\ref{holo..}).
Indeed, in general, the classical Hodge theory does not hold
when the underlying complex manifold is not compact.
However, for  a complete $d$-bounded \ka manifold $(X,\omega)$, by Theorem \ref{Key} we know that $(X,\omega)$ admits the $L^2$ Hodge theory, which provides a possibility to solve the equation (\ref{holo..}) by the same methods as presented in \cite{LRY,LZ}.

\blemma\label{sol}
Let $(X,\omega)$ be a complete $d$-bounded \ka manifold of dimension $n$ and $\varphi\in A^{0,1}(X,T^{1,0}_X)$
be an integral Beltrami differential on $X$ such that 
$\|\varphi\|_{\omega,\infty}<1$.
If $\Omega$ is a smooth $(n,0)$-form in $A^{n,0}_{(2)}(X,\omega)$
such that
\begin{equation}\label{ooh}
  \bp (I+T\varphi)\Omega=0, \,\,\text{with}\,\, \Omega\in\text{Dom} T\varphi\cap A^{n,0}_{(2)}(X,\omega)
\end{equation}
then $\Omega$ gives a solution of the $L^2$ extension equation $($\ref{holo..}$)$.
\elemma
\bproof
First we note that for any $\Omega\in \text{Dom} T\varphi\cap A^{n,0}_{(2)}(X,\omega)$
we have $$\varphi\lrcorner\Omega\in \text{Dom}T=\text{Dom}\p.$$
Let $\Omega_0=(I+T\varphi)\Omega \in Im(I+T\varphi)\subset A^{n,0}_{(2)}(X,\omega)$.
Then $$\Omega=\Omega_0-T\varphi\Omega
  =\Omega_0-\bp ^*\mathbb{G}\p(\varphi\lrcorner\Omega).$$
It follows that
$$\Omega\in \text{Dom}\bp$$
since $\bp\Omega_0= \bp (I+T\varphi)\Omega=0$ and $\bp ^*\mathbb{G}\p(\varphi\lrcorner\Omega)\in \text{Dom}\bp.$
By the injectivity of $I+T\varphi$ (see Lemma \ref{inj}) we have the equation
$$\Omega=(I+T\varphi)^{-1}\Omega_0$$
in $\text{Dom} T\varphi\cap A^{p,q}_{(2)}(X,\omega)$.
We need to show
$$\overline{\partial}\Omega=-\partial(\varphi\lrcorner\Omega).$$
Indeed, from the $L^2$ Hodge theory on $(X,\omega)$, it follows that
\begin{equation}\label{}
  \begin{split}
    \overline{\partial}\Omega &=-\bp\bp ^*\mathbb{G}\p(\varphi\lrcorner\Omega) \\
    &=(\bp ^*\bp- \Delta_{\overline{\partial}})\mathbb{G}\p(\varphi\lrcorner\Omega)\\
    &=(\bp ^*\bp\mathbb{G}-\mathbb{I}+\mathbb{H})\p(\varphi\lrcorner\Omega)\\
    &=-\p(\varphi\lrcorner\Omega)+\bp ^*\bp\mathbb{G}\p(\varphi\lrcorner\Omega).
  \end{split}
\end{equation}
Let $\Phi=\overline{\partial}\Omega+\partial(\varphi\lrcorner\Omega).$
Then it suffices to show $\Phi=0$. We compute
\begin{equation}\label{}
 \begin{split}
   \Phi &=\overline{\partial}\Omega+\partial(\varphi\lrcorner\Omega)\\
     &=\bp^*\bp\mathbb{G}\p(\varphi\lrcorner\Omega)\\
     &= -\bp^*\mathbb{G}\p\bp(\varphi\lrcorner\Omega)\\
     &=-\bp^*\mathbb{G}\p((\bp\varphi)\lrcorner\Omega+\varphi\lrcorner\bp\Omega)\\
     &=-\bp^*\mathbb{G}\p(\frac{1}{2}[\varphi,\varphi]\lrcorner\Omega
     +\varphi\lrcorner(\Phi-\partial(\varphi\lrcorner\Omega)))\\
     &=-\bp^*\mathbb{G}\p(\varphi\lrcorner\Phi)
 \end{split}
\end{equation}
where in the last equality, we have used $\p\Omega=0$, $\p^2=0$ and the formula
\begin{equation*}
  [\varphi,\varphi]\lrcorner\Omega=2\varphi\lrcorner\p\varphi\lrcorner\Omega
  -\p(\varphi\lrcorner\varphi\lrcorner\Omega)-\varphi\lrcorner\varphi\lrcorner\p\Omega.
\end{equation*}
If $\Phi\neq0$ then by Theorem \ref{po} and the condition $\|\varphi\|_{\omega,\infty}<1$ we obtain
\begin{equation*}
  \|\Phi\|\leq\|\varphi\lrcorner\Phi\|\leq\|\varphi\|_{\omega,\infty}\|\Phi\|<\|\Phi\|,
\end{equation*}
a contradiction.Thus we conclude that $\Phi=0$ and
\begin{equation*}
  \overline{\partial}\Omega=-\partial(\varphi\lrcorner\Omega),
\end{equation*}
as desired.
\eproof
Conversely, we have
\blemma \label{hhx}
Let $(X,\omega)$ be a complete $d$-bounded \ka manifold of dimension $n$ and $\varphi\in A^{0,1}(X,T^{1,0}_X)$
be an integral Beltrami differential on $X$ such that 
$\|\varphi\|_{\omega,\infty}<1$.
If the $(n, 0)$-form $\Omega\in A_{(2)}^{n,0}(X,\omega)$ satisfies the equation $($\ref{holo..}$)$,
then $\Omega$ satisfies the equation $($\ref{ooh}$)$.
\elemma
\bproof
Assume that $\Omega\in A_{(2)}^{n,0}(X,\omega)$ satisfies the equation $($\ref{holo..}$)$.
Then we have $\Omega\in\text{Dom} T\varphi\cap A^{n,0}_{(2)}(X,\omega)$
since $\varphi\lrcorner\Omega\in \text{Dom}\p=\text{Dom}T$.
Applying the operator $\bp^*\mathbb{G}$ to (\ref{holo..}), we find
\begin{equation}\label{ssy}
  \bp^*\mathbb{G}\overline{\partial}\Omega=-\bp^*\mathbb{G}\partial(\varphi\lrcorner\Omega).
\end{equation}
From the basic properties of $\mathbb{G}$ and $\mathbb{H}$, we have
\begin{equation}\label{jja}
   \bp^*\mathbb{G}\overline{\partial}\Omega= \bp^*\overline{\partial}\mathbb{G}\Omega=\Delta_{\overline{\partial}}\mathbb{G}=\Omega-\mathbb{H}\Omega.
\end{equation}
Then combining the equations (\ref{ssy}) and ({\ref{jja}}) we obtain
\begin{equation}\label{qqg}
 \mathbb{H}\Omega=\Omega+\bp^*\mathbb{G}\partial(\varphi\lrcorner\Omega)=(I+T\varphi)\Omega.
\end{equation}
Note that $\mathbb{H}\Omega$ is a harmonic $(n,0)$-form on $X$.
Thus it follows that
$$ \bp (I+T\varphi)\Omega=\bp  \mathbb{H}\Omega=0$$
as desired.
\eproof
Summarizing Lemma \ref{sol} and Lemma \ref{hhx}, we obtain 
\btheorem\label{main..}
Let $(X,\omega)$ be a complete $d$-bounded \ka manifold of dimension $n$ and $\varphi\in A^{0,1}(X,T^{1,0}_X)$
be an integral Beltrami differential on $X$ such that 
$\|\varphi\|_{\omega,\infty}<1$.
Then an $(n, 0)$-form $\Omega \in A^{n,0}_{(2)}(X,\omega)$ satisfies the equation $($\ref{holo..}$)$,
that is,
\begin{equation*}
\overline{\partial}\Omega=-\partial(\varphi\lrcorner\Omega),
\,\,\text{with}\,\, \Omega\in \text{Dom}\bp\,\,\text{and}\,\,\varphi\lrcorner\Omega\in \text{Dom}\p,
\end{equation*}
if and only if it satisfies the equation $($\ref{ooh}$)$, that is,
\begin{equation*}
   \bp (I+T\varphi)\Omega=0, \,\,\text{with}\,\, \Omega\in\text{Dom} T\varphi\cap A^{n,0}_{(2)}(X,\omega).
\end{equation*}
 \etheorem

Denote 
$$\mathcal{A}^{n,0}_{(2)}(X,\omega,\varphi):=
Im(I+T\varphi)\subset A^{n,0}_{(2)}(X,\omega).
$$
Following the papers \cite{LRY,LZ} we define the deformation operator 
$\rho_{\omega, \varphi}: \mathcal{A}^{n,0}_{(2)}(X,\omega,\varphi)\rightarrow A^{n,0}(X_{\varphi})$ by 
\begin{equation}
\rho_{\omega, \varphi}(\Omega)=e^{\varphi}\lrcorner(I+T\varphi)^{-1}\Omega.\nonumber
\end{equation}
By Lemma \ref{inj} the deformation operator is well-defined.
Then we have the following main result about extensions of holomorphic canonical forms from $X$ to $X_\varphi$.
\btheorem[=Theorem \ref{.main}]\label{main.}
Let $(X,\omega)$ be a complete $d$-bounded \ka manifold of dimension $n$ and $\varphi\in A^{0,1}(X,T^{1,0}_X)$
be an integral Beltrami differential on $X$ such that 
$\|\varphi\|_{\omega,\infty}<1$.
Then for any holomorphic $(n,0)$-form $\Omega$ in $\mathcal{A}^{n,0}_{(2)}(X,\omega,\varphi)$,
the expression $\rho_{\omega,\varphi}(\Omega)$
defines a holomorphic $(n,0)$-form on $X_\varphi$ with $\rho_{\omega,0}(\Omega)=\Omega$.
\etheorem
\bproof
By the definition of $\mathcal{A}^{n,0}_{(2)}(X,\omega,\varphi)$ we know that for any smooth $(n,0)$-form $\Omega$ in $\mathcal{A}^{n,0}_{(2)}(X,\omega,\varphi)$ there exists an $(n, 0)$-form $\alpha\in\text{Dom} T\varphi\cap A^{n,0}_{(2)}(X,\omega)$
such that $$\Omega=(I+T\varphi)\alpha.$$
By the injectivity of $I+T\varphi$, see Lemma \ref{inj}, we know that
this $\alpha$ is exactly $(I+T\varphi)^{-1}\Omega$, that is,
$$\alpha=(I+T\varphi)^{-1}\Omega.$$
Moreover, if $\Omega$ is holomorphic, then $$\bp (I+T\varphi)\alpha=\bp \Omega=0.$$
By Theorem \ref{main..} and Lemma \ref{holo}
the equation
\begin{equation}
\rho_{\omega, \varphi}(\Omega)=e^{\varphi}\lrcorner(I+T\varphi)^{-1}\Omega=e^{\varphi}\lrcorner\alpha\nonumber
\end{equation}
defines a holomorphic $(n,0)$-form on $X_\varphi$ with $\rho_{\omega,0}(\Omega)=\Omega$, as desired.
\eproof

We remark that theorem \ref{main.} generalizes Theorem 1.1 in \cite{LZ} from the compact to noncompact cases, which is closely related to a famous conjecture due to Siu \cite{Siu98, Siu}, about the invariance of plurigenera for compact
\ka manifolds.

\end{document}